\documentclass{amsart}
\usepackage{xy,amssymb,latexsym}
\theoremstyle{plain}
\newtheorem{theorem}{Theorem}
\newtheorem{corollary}{Corollary}

\newtheorem{lemma}{Lemma}
\newtheorem{definition}{Definition}
\newtheorem{example}{Example}

\begin{document}
\title{A Geometric Approach to Defining Multiplication}
                   
\author{Peter F. McLoughlin and Maria Droujkova}

\email{pmclough@csusb.edu}
\email{droujkova@gmail.com}
\maketitle
The set-theoretic construction of the real numbers in the 1870s 
marked a shift in emphasis from geometric to algebraic reasoning (see [3]). This shift in reasoning may have inadvertently created some problems in the lower-level mathematics curriculum (in the U.S. at least). Let us clarify what we mean by this.
Currently, prospective K-12 math teachers and Science, Technology, Engineering, and Mathematics (STEM) students can graduate from university without ever seeing the following: \\
(1) A rigorous definition of multiplication;\\ 
(2) A proof that two triangles have the same angles if and only if the lengths of their corresponding sides are proportional;\\
(3) A proof regarding the relationship between multiplication and the area of a rectangle.\\
Part of the problem is that multiplication, the area of a rectangle, and similar triangles are not defined independently of each other in the mathematics curriculum.  Here is the definition of similar triangles given in a college algebra text:``Two triangles are similar if the corresponding angles are equal and the lengths of the corresponding sides are proportional.''. Please note this definition is not independent of multiplication (this definition should actually be a Theorem). As the reader can check, many texts define similar triangles this way. If you ask a STEM student or a prospective K-12 math teacher ``what is the area of a rectangle?'' they will more than likely say ``length times width''. Again, this definition is dependent on multiplication (we understand that this is more a problem of semantics than mathematics). The intimate relationships between multiplication, the area of a rectangle, and similar triangles are important because integration rests on limits of areas of rectangles and trigonometry rests on the fact that similar triangles have corresponding sides that are proportional.\\
\\ 
We are suggesting that multiplication, the area of a rectangle, and similar triangles should be defined independently of each other in the mathematics curriculum. After this the intimate relationships between these three things can be proven. In addition, we believe that, all prospective math teachers and STEM students should be exposed to proofs of these intimate relationships before they graduate from university. In this paper we propose a way that this may be accomplished.\\
\\
In this paper we will do the following: (1) show how to geometrically define multiplication, using only basic plane geometry, independently of area and any notion of similar triangles; (2)  prove all the properties of multiplication using only the axioms of plane geometry and the geometric definition of multiplication; (3) explain how the geometric definition of multiplication relates to the area of a right triangle (or rectangle); and (4) explain how by using only the geometric definition of multiplication and the Pythagorean Theorem one can  prove that two triangles have the same angles if and only if the lengths of their corresponding sides are proportional. The interesting and surprising thing, from a pedagogical and/or mathematical point of view, is that all of these results can be proven using only simple geometry (no limits needed). As we shall see, parallel lines in our geometric approach will play a role similar to limits in the standard algebraic approach.\\
\\
The approach we take in this paper is similar to the one taken by Hilbert (see [1]) which in turn can be viewed as a modern interpretation of parts of Euclid's \textit{elements} (see [2]). \textbf {Throughout the paper we will freely assume Euclid's and/or Hilbert's axioms of plane geometry.}\\
\\
\\
\section{the geometric definition of multiplication}
How can we physically interpret real number multiplication? Or to put it another way, is there a simple way to visualize multiplication of any two real numbers?  For example, given any two line segments how could you create a third line segment that is the product of the first two segments? The area model for multiplication is the orthodox physical interpretation that is given to multiplication. Multiplication of two line segments is then viewed as the area of the rectangle they determine. However, for this model to make sense we must relate the area of a rectangle, which is a two-dimensional object, with a line segment, which is a one-dimensional object. In particular, under the area model, how to convert the area of a rectangle to a line segment is not easily visualizable (especially if the numbers being multiplied are not rational).  Moreover, since the area model does not cover signed number multiplication, it is unclear how to multiply signed numbers. \\
We will now propose an alternative physical interpretation of real number multiplication which only uses parallel lines and is based on a simple observation about shadows (or projections). This basic observation is as follows: the hypotenuse of the right triangle determined by an object and its shadow must be parallel to the hypotenuse of any other object and its shadow. Hence, knowing the shadow of one object(we call this object the unit) gives us a way to deduce the shadow of any other object. Now if we replace object with line segment then we are led naturally to the following definition:
\begin{definition} Given two real numbers $a$ and $b$ lay $a$ and $b$ along the y- and x-axis respectively. We define $ab$ (read ``$a$ multiplied by $b$'') to be the x-intercept of the line parallel to the segment $\overline{(b,0),(0,1)}$ which passes through the point $(0,a)$. 
\end{definition}
In this definition, $(0,1)$ is our unit,  $(b,0)$ is our unit shadow determined by $b$, and segment $\overline{(b,0),(0,1)}$ is our unit hypotenuse determined by $b$. 
Please note this geometric definition of multiplication only uses parallel lines and  \textbf{does not presuppose any knowledge of similar triangles.}
\textbf {Moreover, to avoid circular reasoning all the proofs and examples based on the definition will only use congruency arguments.} 
This definition of multiplication is equivalent to the one used by Hilbert(see [1] page 47).
\begin{example}
Use Definition 1 to multiply 2 and 4 then use congruent triangles to show that 2(4)=4+4.
\end{example}
\textbf{Solution:}\\
Step 1:
\[
	\xy
	(16,0)*{\bullet};(0,0)*{\bullet} **\dir{-}; (16,-3)*{4};(0,-3)*{0};
	\endxy
\]
Step 2:
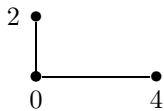
\[
	\xy
	(0,8)*{\bullet};(0,0)*{\bullet} **\dir{-}; (-3,8)*{2};(0,-3)*{0};
	(16,0)*{\bullet};(0,0)*{\bullet} **\dir{-}; (16,-3)*{4};
	\endxy
\]
Step 3:
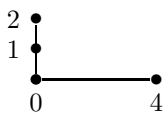
\[
	\xy
	(0,8)*{\bullet};(0,0)*{\bullet} **\dir{-}; (-3,8)*{2};(0,-3)*{0};
	(16,0)*{\bullet};(0,0)*{\bullet} **\dir{-}; (16,-3)*{4};
	(-3,4)*{1};(0,4)*{\bullet};
	\endxy
\]
Step 4: Connect (0,1) to (4,0).\\
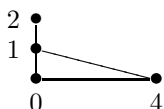
\[
	\xy
	(0,8)*{\bullet};(0,0)*{\bullet} **\dir{-}; (-3,8)*{2};(0,-3)*{0};
	(16,0)*{\bullet};(0,0)*{\bullet} **\dir{-}; (16,-3)*{4};
	(-3,4)*{1};(0,4)*{\bullet};(0,4)*{\bullet};(16,0)*{\bullet} **\dir{-};
	\endxy
\]
Step 5: Now draw the line passing through (0,2) which is parallel to the segment $\overline{(0,1),(4,0)}$. By Definition 1, the x-intercept of this line is 2(4).
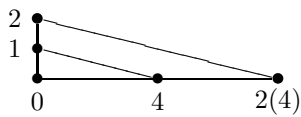
\[
	\xy
	(0,8)*{\bullet};(32,0)*{\bullet} **\dir{-};
	(32,0)*{\bullet};(0,0)*{} **\dir{-}; 
	(0,0)*{};(0,8)*{\bullet} **\dir{-}; 
	(0,4)*{\bullet};(16,0)*{\bullet} **\dir{-}; 
	(-3,8)*{2};(16,-3)*{4};(32,-3)*{2(4)};(-3,4)*{1};(0,-3)*{0};(0,0)*{\bullet};
	\endxy
\]
Step 6: Refer to the figure below. There exists a point, A, on $\overline{(0,2),(2(4),0)}$ such that $\overline{\left\{A,(4,0)\right\}}$ is parallel to $\overline{(0,1),(0,2)}$. By design $\{(0,1),(0,2),(4,0), A\}$ form the vertices of a parallelogram (why?). This implies $A=(4,1)$.
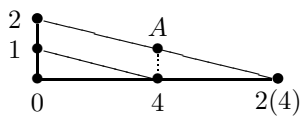
\[
	\xy
	(16,0)*{};(16,4)*{\bullet} **\dir{..};
	(0,8)*{\bullet};(32,0)*{\bullet} **\dir{-};
	(32,0)*{\bullet};(0,0)*{} **\dir{-}; 
	(0,0)*{};(0,8)*{\bullet} **\dir{-}; 
	(0,4)*{\bullet};(16,0)*{\bullet} **\dir{-}; 
	(-3,8)*{2};(16,-3)*{4};(32,-3)*{2(4)};(-3,4)*{1};(0,-3)*{0};(0,0)*{\bullet};(16,7)*{A}
	\endxy
\]
Using angle-side-angle we can deduce that the triangle determined by $\left\{(0,1), (0,0), (4,0)\right\}$ is congruent to the triangle determined by $\left\{(4,0), A, (2(4),0)\right\}$. Hence we must have 2(4)=4+4.
\begin{example}
Use Definition 1 to multiply -2 and -4 then use congruent triangles to show that $(-2)(-4)=2(4)$.
\end{example}
\textbf{Solution:}\\
Step 1:
\[
	\xy
	(-16,0)*{\bullet};(0,0)*{\bullet} **\dir{-}; (-16,-3)*{-4};(0,-3)*{0};
	\endxy
\]
Step 2:
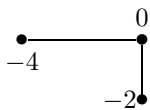
\[
	\xy
	(0,-8)*{\bullet};(0,0)*{\bullet} **\dir{-}; (-3,-8)*{-2};(0,3)*{0};
	(-16,0)*{\bullet};(0,0)*{\bullet} **\dir{-}; (-16,-3)*{-4};
	\endxy
\]
Step 3:
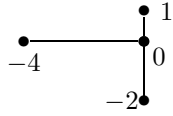
\[
	\xy
	(0,-8)*{\bullet};(0,0)*{\bullet} **\dir{-}; (-3,-8)*{-2};(2,-2)*{0};
	(-16,0)*{\bullet};(0,0)*{\bullet} **\dir{-}; (-16,-3)*{-4};
	(3,4)*{1};(0,4)*{\bullet};(0,0)*{\bullet} **\dir{-};
		\endxy
\]
Step 4: Connect (0,1) to (-4,0).\\
\[
	\xy
	(0,-8)*{\bullet};(0,0)*{\bullet} **\dir{-}; (-3,-8)*{-2};(2,-2)*{0};
	(-16,0)*{\bullet};(0,0)*{\bullet} **\dir{-}; (-16,-3)*{-4};
	(3,4)*{1};(0,4)*{\bullet};(0,4)*{\bullet};(-16,0)*{\bullet} **\dir{-};
	(0,4)*{\bullet};(0,0)*{\bullet} **\dir{-};
	\endxy
\]
Step 5: Now draw the line through (0,-2) which is parallel to the segment $\overline{(0,1),(-4,0)}$. By Definition 1, the x-intercept of this line is (-2)(-4).\\
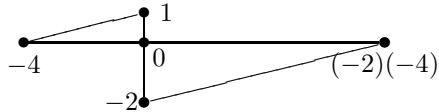
\[
	\xy
	(0,-8)*{\bullet};(32,0)*{\bullet} **\dir{-};
	(32,0)*{\bullet};(-16,0)*{} **\dir{-}; 
	(0,4)*{};(0,-8)*{\bullet} **\dir{-}; 
	(0,4)*{\bullet};(-16,0)*{\bullet} **\dir{-}; 
	(-3,-8)*{-2};(-16,-3)*{-4};(32,-3)*{(-2)(-4)};(3,4)*{1};(2,-2)*{0};(0,0)*{\bullet};
	\endxy
\]
Step 6: In the figure below one can easily show, using congruent triangles and example 1, that (-2)(-4)=2(4).\\
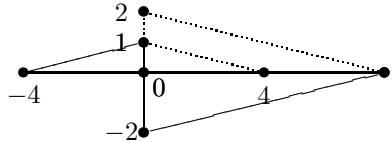
\[
	\xy
	(0,-8)*{\bullet};(32,0)*{\bullet} **\dir{-};
	(32,0)*{\bullet};(-16,0)*{} **\dir{-}; 
	(0,4)*{};(0,-8)*{\bullet} **\dir{-}; 
	(0,4)*{\bullet};(-16,0)*{\bullet} **\dir{-}; 
	(-3,-8)*{-2};(-16,-3)*{-4};(32,-3)*{};(-3,4)*{1};(2,-2)*{0};(0,0)*{\bullet};
		(0,8)*{\bullet};(32,0)*{\bullet} **\dir{..};
	(32,0)*{\bullet};(0,0)*{} **\dir{..}; 
	(0,0)*{};(0,8)*{\bullet} **\dir{..}; 
	(0,4)*{\bullet};(16,0)*{\bullet} **\dir{..}; 
	(-3,8)*{2};(16,-3)*{4};(-3,4)*{1};(2,-2)*{0};(0,0)*{\bullet};
	\endxy
\]
\\
\begin{example}
Use Definition 1 to multiply 2 and -4 then use congruent triangles to show that 2(-4)=-(2(4))=-8.
\end{example}
\textbf{Solution:}\\
Step 1:
\[
	\xy
	(-16,0)*{\bullet};(0,0)*{\bullet} **\dir{-}; (-16,-3)*{-4};(0,-3)*{0};
	\endxy
\]
Step 2:
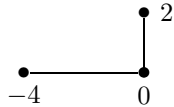
\[
	\xy
	(0,8)*{\bullet};(0,0)*{\bullet} **\dir{-}; (3,8)*{2};(0,-3)*{0};
	(-16,0)*{\bullet};(0,0)*{\bullet} **\dir{-}; (-16,-3)*{-4};
	\endxy
\]
Step 3:
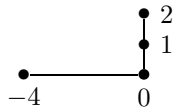
\[
	\xy
	(0,8)*{\bullet};(0,0)*{\bullet} **\dir{-}; (3,8)*{2};(0,-3)*{0};
	(-16,0)*{\bullet};(0,0)*{\bullet} **\dir{-}; (-16,-3)*{-4};
	(3,4)*{1};(0,4)*{\bullet};
	\endxy
\]
Step 4: Connect (0,1) to (-4,0).\\
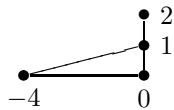
\[
	\xy
	(0,8)*{\bullet};(0,0)*{\bullet} **\dir{-}; (3,8)*{2};(0,-3)*{0};
	(-16,0)*{\bullet};(0,0)*{\bullet} **\dir{-}; (-16,-3)*{-4};
	(3,4)*{1};(0,4)*{\bullet};(0,4)*{\bullet};(-16,0)*{\bullet} **\dir{-};
	\endxy
\]
Step 5: Now draw the line passing through (0,2) which is parallel to the segment $\overline{(0,1),(-4,0)}$. By Definition 1, the x-intercept of this line is 2(-4).
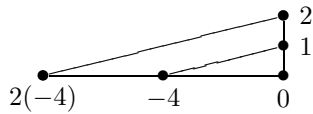
\[
	\xy
	(0,8)*{\bullet};(-32,0)*{\bullet} **\dir{-};
	(-32,0)*{\bullet};(0,0)*{} **\dir{-}; 
	(0,0)*{};(0,8)*{\bullet} **\dir{-}; 
	(0,4)*{\bullet};(-16,0)*{\bullet} **\dir{-}; 
	(3,8)*{2};(-16,-3)*{-4};(-32,-3)*{2(-4)};(3,4)*{1};(0,-3)*{0};(0,0)*{\bullet};
	\endxy
\]
Step 6: Connect (-4,0) to (-4,1).\\
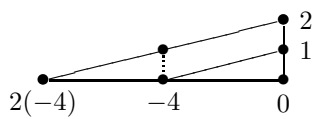
\[
	\xy
	(-16,0)*{};(-16,4)*{\bullet} **\dir{..};
	(0,8)*{\bullet};(-32,0)*{\bullet} **\dir{-};
	(-32,0)*{\bullet};(0,0)*{} **\dir{-}; 
	(0,0)*{};(0,8)*{\bullet} **\dir{-}; 
	(0,4)*{\bullet};(-16,0)*{\bullet} **\dir{-}; 
	(3,8)*{2};(-16,-3)*{-4};(-32,-3)*{2(-4)};(3,4)*{1};(0,-3)*{0};(0,0)*{\bullet};
	\endxy
\]
In the figure above the triangle determined by the three points (0,1), (0,0) and (-4,0) is congruent to the triangle determined by (-4,0), (-4,1) and (2(-4),0). Hence we must have $2(-4)=-(4+4)=-8$.
It is left as an exercise for the reader to show, using definition 1, that $(-2)4=2(-4)$.\\
The foregoing examples provide the simple visual insight necessary in order to prove the general rules about sign number multiplication. Furthermore, example 1 provides the intuitive insight necessary to show that our multiplication definition reduces to repeated addition when restricted to whole numbers.
\section{Extending students' understanding of multiplication }
In this section we show that our geometric definition of multiplication naturally extends the students' understanding of multiplication from the whole numbers (where multiplication can be viewed as repeated addition) to the real numbers.
\\
\\
\begin{theorem} 
If $a\neq 0$ is a whole number and $b$ any real number then $ab=\sum^{a}_{i=1} b$.
\end{theorem}
\begin{proof}
By definition of multiplication $\overline{(b,0),(0,1)}$ is parallel to $\overline{(0,a),(ab,0)}$. Partition  $a$ into units. Each unit of this partition can be used to determine a right triangle along the segment $\overline{(0,a),(ab,0)}$ (refer to figure below). Moreover, by angle-side-angle, each of these triangles is congruent to the triangle determined by the points: (0,0), (0,1) and (b,0). It follows we must have $ab=\sum^{a}_{i=1}b$.
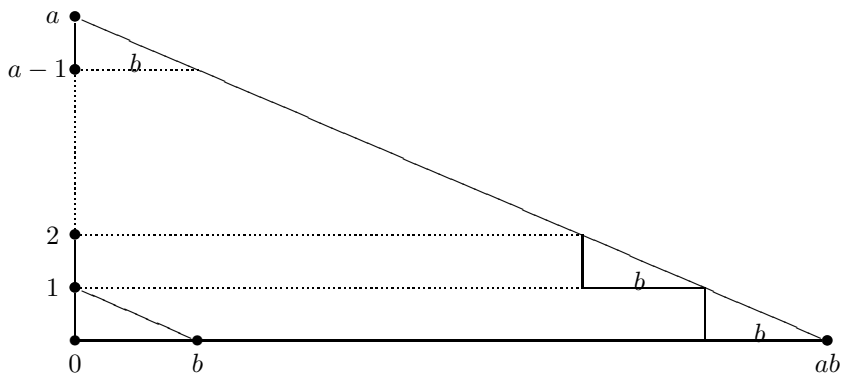
\[
	\xy
	(0,43)*{\bullet};(100,0)*{\bullet} **\dir{-};
	(0,43)*{\bullet};(0,36)*{\bullet} **\dir{-};
	(0,36)*{\bullet};(0,14)*{\bullet} **\dir{.};
	(0,14)*{\bullet};(0,7)*{\bullet} **\dir{-};
	(100,0)*{\bullet};(0,0)*{} **\dir{-}; 
	(0,0)*{};(0,7)*{\bullet} **\dir{-}; 
	(0,7)*{\bullet};(16.275,0)*{\bullet} **\dir{-};
	(0,36)*{\bullet};(16.275,36)*{} **\dir{.};
	(83.725,0)*{};(83.725,7)*{} **\dir{-}; 
	(67.45,7,0)*{};(67.45,14)*{} **\dir{-}; 
	(67.45,7)*{};(83.725,7)*{} **\dir{-};
	(0,7)*{};(67.45,7)*{} **\dir{.};
	(0,14)*{};(67.45,14)*{} **\dir{.};
	(-3,43)*{a};(-5,36)*{a-1 };(-3,14)*{2};(-3,7)*{1};(8,37)*{b};(16.275,-3)*{b};(100,-3)*{ab};(91,1)*{b};(75,8)*{b};(0,-3)*{0};(0,0)*{\bullet};
	\endxy
\]
\end{proof}
\section{Multiplication of signed numbers}
In this section we show that multiplication with signed numbers is completely natural. In particular, proving that a negative number times a negative number is a positive number (something most beginning students accept on faith) follows immediately from our definition of multiplication. In addition, our proof, unlike the traditional proof, does not require the use of the distributive property (for the conventional ways of teaching multiplication of signed numbers see [6]).
\begin{theorem}
For any positive real numbers a and b the following are true:\\
1.) $(-a)(-b)=ab$;\\
2.) $a(-b)=a(-b)=-(ab)$.
\end{theorem}
\begin{proof}
We prove one, the other case follows  mutatis mutandis. In the figure below the two smaller triangles are congruent by side-angle-side. Moreover, the two larger triangles are also congruent by side-angle-side. This implies,  $\overline{(b,0),(0,1)} \parallel \overline{C,(0,a)}$ if and only if  $\overline{(-b,0),(0,1)} \parallel  \overline{C,(0,-a)}$. Hence, by definition of multiplication, we must have  $(-a)(-b)=ab$.
\\
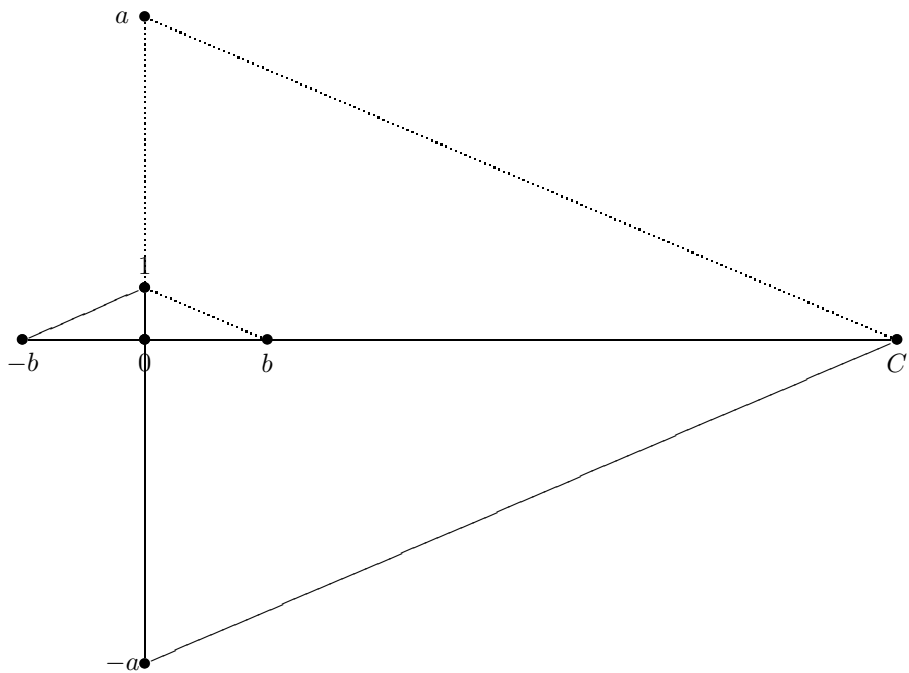
\[
	\xy
	(0,-43)*{\bullet};(100,0)*{\bullet} **\dir{-};
	(100,0)*{\bullet};(0,0)*{} **\dir{-}; 
	(0,0)*{};(0,-43)*{\bullet} **\dir{-}; 
	(0,7)*{\bullet};(-16.275,0)*{\bullet} **\dir{-}; 
	(0,7)*{\bullet};(0,0)*{\bullet} **\dir{-}; 
	(0,0)*{\bullet};(-16.275,0)*{\bullet} **\dir{-}; 
	(0,0)*{\bullet};(16.275,0)*{\bullet} **\dir{-};
	(0,7)*{\bullet};(16.275,0)*{\bullet} **\dir{.};
	(0,43)*{\bullet};(100,0)*{\bullet} **\dir{.};
	(0,0)*{\bullet};(-16.275,0)*{\bullet} **\dir{-};
	(0,7)*{\bullet};(0,43)*{\bullet} **\dir{.};  
	(-3,-43)*{-a};(-16.275,-3)*{-b};(100,-3)*{C};
	(-3,43)*{a};(16.275,-3)*{b};(100,-3)*{};(0,10)*{1};(0,-3)*{0};(0,0)*{\bullet};
	\endxy
\]
\end{proof}
\section{The relative size of Multiplied numbers}
In this section we use our definition of multiplication to visually show that the product of two positive real numbers may not always be larger than the numbers being multiplied.
\begin{theorem} 
If $0<a<1$ and $b>0$ then $ab<b$.
\end{theorem}
\begin{proof}
Refer to the figure below. Proof follows directly from Definition 1.
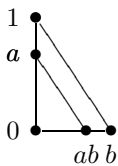
\[
	\xy
	(0,15)*{\bullet};(0,0)*{} **\dir{-};
	(0,0)*{\bullet};(10,0)*{\bullet} **\dir{-}; 
	(0,10)*{\bullet};(-3,10)*{a} ;
	(0,15)*{\bullet};(10,0)*{}**\dir{-};
	(0,10)*{};(6.6666,0)*{}**\dir{-};
	(6.6666,-3)*{ab}; (-3,10)*{a};(-3,15)*{1};(10,-3)*{b};(6.6666,0)*{\bullet};(-3,0)*{0}
	\endxy
\]

\end{proof}
\section{The Inverse of a real number}
\begin{theorem} 
If $a\neq 0$ is a real number then there exists a unique real number $b\neq 0$ such that $ab=1$. Moreover, we call $b$ the inverse of $a$ and denote it by $\frac{1}{a}$.
\end{theorem}
\begin{proof}
Consider the unique line that is parallel to $\overline{(0,a),(1,0)}$ and passes through $(0,1)$. Let $(b,0)$ be the x-intercept of this line (refer to figure below). By definition of multiplication we have $ab=1$. 
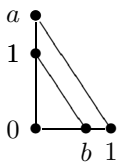
\[
	\xy
	(0,15)*{\bullet};(0,0)*{} **\dir{-};
	(0,0)*{\bullet};(10,0)*{\bullet} **\dir{-}; 
	(0,10)*{\bullet};(-3,10)*{1} ;
	(0,15)*{\bullet};(10,0)*{}**\dir{-};
	(0,10)*{};(6.6666,0)*{}**\dir{-};
	(6.6666,-3)*{b}; (-3,10)*{1};(-3,15)*{a};(10,-3)*{1};(6.6666,0)*{\bullet};(-3,0)*{0}
	\endxy
\]
\end{proof}
\section{A relationship between areas of Triangles and parallel lines}
The following Lemma is Propostion 37 in book one of Euclid's \textit{elements}.
\begin{lemma}
 $A(\Delta ABC)=A(\Delta ABC_{1})$ if and only if $CC_{1} \parallel AB$
\end{lemma}
\begin{proof} 
Suppose $CC_{1}\parallel AB$. Refer to the diagram below. Let $C^{'}$ be the unique point where $CC^{'}= AB$ and $CC^{'}\parallel AB$.
Similarly, let $C_{1}^{'}$ be the unique point where $C_{1}C_{1}^{'}= AB$ and $C_{1}C_{1}^{'}\parallel AB$.
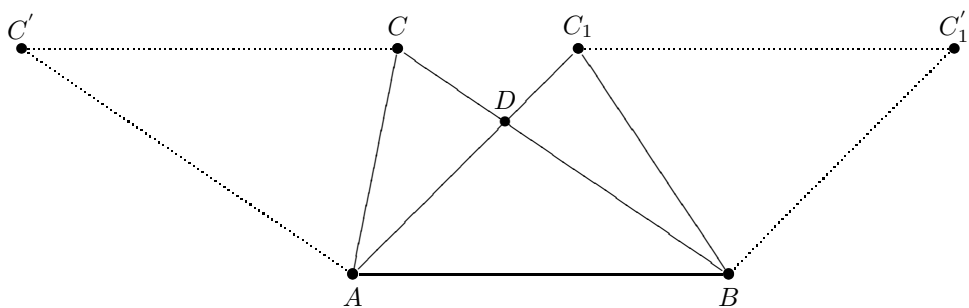
\[
	\xy
	(0,0)*{\bullet};(50,0)*{\bullet}**\dir{-};(50,-3)*{B};(0,-3)*{A};
	(6,30)*{\bullet};(0,0)*{\bullet}**\dir{-};(6,30)*{\bullet};(50,0)*{\bullet}**\dir{-};
	(30,30)*{\bullet};(0,0)*{\bullet}**\dir{-};(30,30)*{\bullet};(50,0)*{\bullet}**\dir{-};
	(6,33)*{C};(30,33)*{C_{1}};(-44,33)*{C^{'}};(80,33)*{C_{1}^{'}};
	(-44,30)*{\bullet};(0,0)*{\bullet}**\dir{.};(80,30)*{\bullet};(50,0)*{\bullet}**\dir{.};
		(-44,30)*{\bullet};(6,30)*{\bullet}**\dir{.};(80,30)*{\bullet};(30,30)*{\bullet}**\dir{.};
		(20.27,20.27)*{\bullet};(20.27,23.27)*{D};
\endxy
\]
By side-angle-side we have  $\Delta C^{'}CA=\Delta ABC$ and $\Delta ABC_{1}=\Delta C_{1}^{'}C_{1}B$. This implies $ C^{'}A=CB$ and 
$AC_{1}=C_{1}^{'}B$.
Now by side-side-side we have $\Delta C^{'}C_{1}A=\Delta CC_{1}^{'}B$. By referring to the diagram above we see that
$A(\Delta C^{'}C_{1}A)-A(\Delta CDC_{1})+A(\Delta ADB)=A(\Delta C^{'}CA)+A(\Delta ABC)$ and $A(\Delta CC_{1}^{'}B )-A(\Delta CDC_{1})+A(\Delta ADB)=A(\Delta C_{1}^{'}C_{1}B)+A(\Delta ABC_{1})$. Hence, from the foregoing statements we are able to deduce that $A(\Delta ABC)=A(\Delta ABC_{1})$.
\\
Conversely, suppose that  $A(\Delta ABC)=A(\Delta ABC_{1})$. Let $\tilde{C_{1}}$ be the unique point such that $ C\tilde{C_{1}}\parallel AB$ and $C\tilde{C_{1}}\perp \tilde{C_{1}}C_{1}$ or $\tilde{C_{1}}=C_{1}$. By using the first part of the proof we have $A(\Delta ABC)=A(\Delta AB\tilde{C_{1}})$. It follows we must have $A(\Delta ABC_{1})=A(\Delta AB\tilde{C_{1}})$ which is only possible if  $\tilde{C_{1}}=C_{1}$. Therefore, we have shown that $CC_{1} \parallel AB$.
\end{proof}
\begin{theorem}
Let $\Delta A B C$ and $\Delta A B_{1}C_{1}$ be two right triangles with $\angle BAC = \angle BAC_{1}= 90$. The areas of the two triangles are equal if and only if $BC_{1}$ is parallel to $B_{1}C$. 
\end{theorem}
\begin{proof}
 Refer to figure below. Note $A(\Delta A B C)=A(\Delta A B_{1}C_{1}) \Leftrightarrow A(\Delta BC_{1}C)=A(\Delta B_{1}BC_{1})$. Also by Lemma 4, $A(\Delta BC_{1}C)=A(\Delta B_{1}BC_{1})  \Leftrightarrow BC_{1}\parallel B_{1}C$. Hence claim follows. 
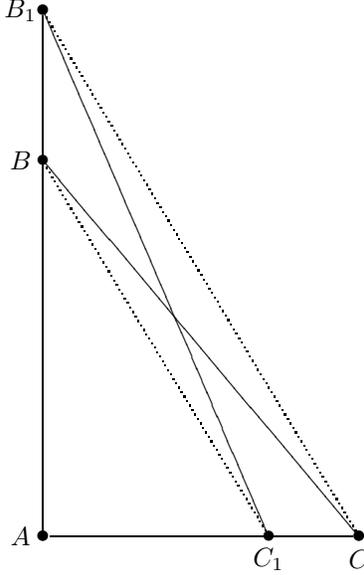
\[
 \xy
	(0,70)*{\bullet};(42,0)*{\bullet} **\dir{..};
	(0,50)*{\bullet};(30,0)*{\bullet} **\dir{..};
	(0,0)*{\bullet};(42,0)*{\bullet} **\dir{-};
	(0,0)*{\bullet};(0,70)*{\bullet} **\dir{-};
	(0,50)*{};(42,0)*{} **\dir{-};
	(0,70)*{};(30,0)*{} **\dir{-};
	(-3,0)*{A};(-3,50)*{B};(42,-3)*{C};(-3,70)*{B_{1}};(30,-3)*{C_{1}};
	\endxy
\]
\end{proof}
\begin{corollary}
Let $\Delta A B C$ and $\Delta A B_{1}C_{1}$ be two triangles with $\angle BAC = \angle BAC_{1}$. The areas of the two triangles are equal if and only if $BC_{1}$ is parallel to $B_{1}C$. 
\end{corollary}
\begin{proof}
Note there is nothing special about the angle being 90 degrees in the proof of Theorem 5.
\end{proof}

\section{A relationship between the area of a right triangle and multiplication}
A moment's reflection should convince the reader that the area of a triangle defines an equivalence relation on the set of all right triangles (this is a fancy way of saying it partitions the set of right triangles).\\
\textbf{Question:} Given a set containing all right triangles of the same area (an equivalence class) is it possible to assign it a real number in a well-defined way? \\
\textbf{Answer:} Yes. In each equivalence class there must be a right triangle that has the unit (some agreed upon length)  as one of its legs (why?). We will take the length of the other leg of this triangle as the real number to assign to the equivalence class.\\
\textbf{Question:} Given any right triangle how can we find the real number that is assigned to its area? \\
\textbf{Answer:} We can use Theorem 5 to do this. In particular, given any right triangle, Theorem 5 provides a way to construct a second right triangle with the following properties: (1) it has the same area as the first triangle, and (2) the unit is one of its legs.
The construction is a follows: let $(0,a)$ and $(b,0)$ be the two legs of any right triangle then the legs of the second triangle can be taken as $(0,1)$ and the x-intercept of the line parallel to the segment $\overline{(b,0),(0,1)}$ which passes through the point $(0,a)$ (note: this is precisely Definition 1). 
\begin{definition}
For positive real numbers $a$ and $b$ we define $T_{ab}$ to be a right triangle with legs $a$ and $b$. Moreover we will denote the area of $T_{ab}$ by $A(T_{ab})$.
\end{definition}
\begin{theorem}
$A(T_{ab})=A(T_{a_{1}b_{1}})$ if and only if $ab=a_{1}b_{1}$.
\end{theorem}
\begin{proof}
By definition of multiplication and Theorem 5 we have:\\
(1) $A(\Delta{(0,0),(0,1),(ab,0)})=A(T_{ab})$;\\
(2) $A(\Delta{(0,0),(0,1),(a_{1}b_{1},0)})=A(T_{a_{1}b_{1}})$.\\
 Now if $ab=a_{1}b_{1}$ then, by (1) and (2), we must have $A(T_{ab})=A(T_{a_{1}b_{1}})$. Conversely, suppose $A(T_{ab})=A(T_{a_{1}b_{1}})$. There exists a positive real number $b_{2}$ such that $ab=a_{1}b_{2}$. It follows $A(T_{ab})=A(T_{a_{1}b_{2}})$ and $A(T_{ab})=A(T_{a_{1}b_{1}})$ which imply $A(T_{a_{1}b_{2}})=A(T_{a_{1}b_{1}})$. Hence, we must have $b_{2}=b_{1}$ and $ab=a_{1}b_{1}$.
\end{proof}
\begin{corollary} 
Multiplication is commutative on positive real numbers.
\end{corollary}
\begin{proof}
Note for real numbers $a$ and $b$ we have $A(T_{ab})=A(T_{ba})$. By Theorem 6, we must have $ab=ba$.
\end{proof}
\section{The Commutative Property of Multiplication}
\begin{theorem}
Multiplication of real numbers is commutative.
\end{theorem}
\begin{proof}
Claim follows from Corollary 2 and Theorem 2.
\end{proof}
\section{The Associative Property of Multiplication}
\begin{definition}
For real numbers a,b and c we define $abc=a(bc)$.
\end{definition}
\begin{lemma} 
If a, b and c are real numbers then $a(cb)=c(ab)$.
\end{lemma}
\begin{proof}
We will prove the case when $a>0, b>0$ and $c>0$. The general case will then follow by Theorem 2.
Refer to figure below. By definition of multiplication we have 
$\overline{(cb,0),(0,c)}  \parallel \overline{(0,1),(b,0)}$ and $\overline{(ab,0),(0,a)}  \parallel \overline{(0,1),(b,0)}$. This implies $\overline{(cb,0),(0,c)} \parallel \overline{(ab,0),(0,a)}$. By Theorem 5,  $\overline{(cb,0),(0,c)} \parallel \overline{(ab,0),(0,a)}$ if and only if $A(T_{a(cb)})=A(T_{c(ab)})$. Moreover, by Theorem 6, $A(T_{a(cb)})=A(T_{c(ab)})$ if and only if $a(cb)=c(ab)$. Hence we must have $a(cb)=c(ab)$.
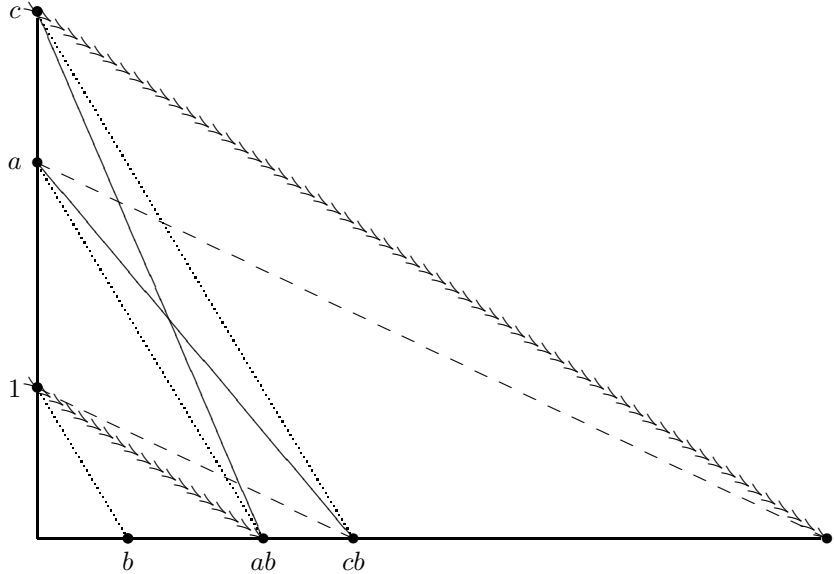
\[
 \xy
	(0,70)*{\bullet};(42,0)*{\bullet} **\dir{..};(0,20)*{\bullet};(12,0)*{\bullet};
	(0,50)*{\bullet};(30,0)*{\bullet} **\dir{..};
	(0,20)*{\bullet};(12,0)*{\bullet} **\dir{..};
	(0,0)*{};(105,0)*{\bullet} **\dir{-};
	(0,0)*{};(0,70)*{\bullet} **\dir{-};
	(0,50)*{};(42,0)*{} **\dir{-};
	(0,70)*{};(30,0)*{} **\dir{-};
	(0,20)*{};(30,0)*{} **\dir{>};
	(0,20)*{};(42,0)*{} **\dir{--};
	(0,70)*{};(105,0)*{} **\dir{>};
	(0,50)*{};(105,0)*{} **\dir{--};
	(-3,50)*{a};(42,-3)*{cb};(-3,70)*{c};(30,-3)*{ab};(12,-3)*{b};(-3,20)*{1};
	\endxy
\]
\end{proof}
\begin{theorem} 
If a, b and c are real numbers then $a(bc)=(ab)c$.
\end{theorem}
\begin{proof}
By Theorem 7, we have $a(bc)=a(cb)$ and $(ab)c=c(ab)$. This implies $a(bc)=(ab)c$ if and only if $a(cb)=c(ab)$. By Lemma 5 $a(cb)=c(ab)$. Hence we must have $a(bc)=(ab)c$.
\end{proof}
\section{The Distributive Property of Multiplication}
\begin{theorem}
$(b+c)a= ba+ca$ for all real numbers a,b and c.
\end{theorem} 
\begin{proof}
We prove the case when $a>0$, $b>0$ and $c>0$. The other cases follow mutatis mutandis. By definiton of multiplication we have $\overline{(0,c),(ca,0)} \parallel \overline{(0,b+c),((b+c)a,0)}$ (Why?). Refer to the figure below. By angle-side-angle we have triangle $\{(0,0),(0,c), (ca,0)\}$ congruent to triangle $\{(ba,0),A, ((b+c)a,0)\}$. This implies $(b+c)a-ba=ca$.
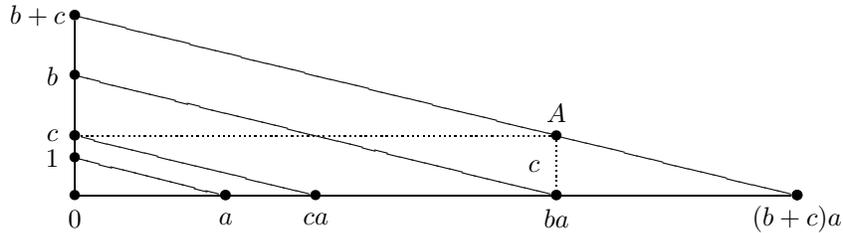
\[
	\xy
	(0,5)*{\bullet};(0,16)*{\bullet};(0,24)*{\bullet};(20,0)*{\bullet};(64,0)*{\bullet};(96,0)*{\bullet};
	(-3,5)*{1};(-3,16)*{b};(-5,24)*{b+c};(20,-3)*{a};(64,-3)*{ba};(96,-3)*{(b+c)a};
	(0,0)*{};(96,0)*{\bullet}**\dir{-};
	(0,5)*{};(20,0)*{\bullet}**\dir{-};
	(0,16)*{};(64,0)*{\bullet}**\dir{-};
	(0,24)*{};(96,0)*{\bullet}**\dir{-};
	(0,0)*{};(0,24)*{\bullet}**\dir{-};
	(64,0)*{};(64,8)*{}**\dir{.};
	(61,4)*{c};(0,-3)*{0};(0,0)*{\bullet};
	(0,8)*{\bullet};(64,8)*{\bullet}**\dir{.};
	(-3,8)*{c};(64,11)*{A};
	(0,8)*{\bullet};(32,0)*{\bullet}**\dir{-};(32,-3)*{ca}
	\endxy
\]
\end{proof}

\section{Multiplication with fractions}
We have choosen to write this section algebraically using the results from the previous sections, for two reasons. First, because the algebra is elegant enough, and second, to illustrate how the geometric definition of multiplication is in accord with the traditional algebraic definition and algorithms.
\begin{lemma} 
For nonzero real numbers $b_{1}$ and $b_{2}$ we have $\frac{1}{b_{1}}\cdot\frac{1}{b_{2}}=\frac{1}{b_{1}\cdot b_{2}}$.
\end{lemma}
\begin{proof}
Observe that $1=(b_{1}\cdot b_{2})(\frac{1}{b_{1}\cdot b_{2}})$. Moreover, by the associative and commutative properties of multiplication we have: $ (b_{1}\cdot b_{2})(\frac{1}{b_{1}}\cdot\frac{1}{b_{2}})=(b_{1}(\frac{1}{b_{1}}))\cdot(b_{2}(\frac{1}{b_{2}}))=1$. Hence by uniqueness of inverses we must $\frac{1}{b_{1}} \cdot\frac{1}{b_{2}}=\frac{1}{b_{1}\cdot b_{2}}$.
\end{proof}
\begin{theorem} 
For real numbers $a_{1}$, $a_{2}$, $b_{1} \neq 0$ and $b_{1} \neq 0$ we have $\frac{a_{1}}{b_{1}}\cdot\frac{a_{2}}{b_{2}}=\frac{a_{1}\cdot a_{2}}{b_{1}\cdot b_{2}}$.
\end{theorem}
\begin{proof}
By the associative and commutative properties of multiplication we have: $\frac{a_{1}}{b_{1}}\cdot\frac{a_{2}}{b_{2}}=(a_{1}(\frac{1}{b_{1}}))(a_{2}(\frac{1}{b_{2}}))=(a_{1}\cdot a_{2})(\frac{1}{b_{1}}\cdot\frac{1}{b_{2}})$. Moreover by Lemma 6 we have: $(a_{1}\cdot a_{2})(\frac{1}{b_{1}}\cdot\frac{1}{b_{2}})=(a_{1}\cdot a_{2})(\frac{1}{b_{1}\cdot{b_{2}}})=\frac{a_{1}\cdot a_{2}}{b_{1}\cdot b_{2}}$. Hence we must have  $\frac{a_{1}}{b_{1}}\cdot\frac{a_{2}}{b_{2}}=\frac{a_{1}\cdot a_{2}}{b_{1}\cdot b_{2}}$.
\end{proof} 
\begin{corollary} 
For real numbers $a$, $b \neq 0$ and $k\neq 0$ we have $\frac{a}{b}=\frac{ak}{bk}$
\end{corollary}
\begin{proof}
By Theorem 10 we have: $\frac{ak}{bk}=\frac{a}{b}\cdot \frac{k}{k}=\frac{a}{b}\cdot 1=\frac{a}{b}$
\end{proof}
\begin{corollary} 
For real numbers $a_{1}$, $a_{2}$, $b_{1} \neq 0$ and $b_{1} \neq 0$ we have $\frac{a_{1}}{b_{1}}=\frac{a_{2}}{b_{2}}$ if and only if $a_{1}b_{2}=a_{2}b_{1}$.
\end{corollary}
\begin{proof}
By Theorem 10 and Corollary 3 we have $a_{1}b_{2}=a_{2}b_{1}\Leftrightarrow (\frac{1}{b_{1}b_{2}})(a_{1}b_{2})=(\frac{1}{b_{1}b_{2}})(a_{2}b_{1})\Leftrightarrow \frac{a_{1}b_{2}}{b_{2}b_{1}}=\frac{b_{1}a_{2}}{b_{2}b_{1}}\Leftrightarrow \frac{a_{1}}{b_{1}}=\frac{a_{2}}{b_{2}}$
\end{proof}
\section{Similar triangles}
\begin{definition}
We call two triangles similar if they have the same angles.
\end{definition}
\begin{lemma}
If $\Delta ABC$ is a right triangle with sides $a\leq b< c$ then any triangle $\Delta A_{1}B_{1}C_{1}$ with sides $a_{1}\leq b_{1}\leq c_{1}$ is similar to $\Delta ABC$ if and only if there exists a $k>0$ such that $a=ka_{1}, b=kb_{1}$ and $c=kc_{1}$.
\end{lemma}
\begin{proof}
$(\Rightarrow)$\\
Suppose $\Delta ABC$ is similar to $\Delta A_{1}B_{1}C_{1}$. Without-loss-of-generality we may assume $1\leq a_{1}<a$. There exists a $k>1$ such that $a=ka_{1}$. Mark-off one unit along $a_{1}$. Next create a right triangle similar to $\Delta ABC$ with legs 1 and r (refer to diagram below).\\
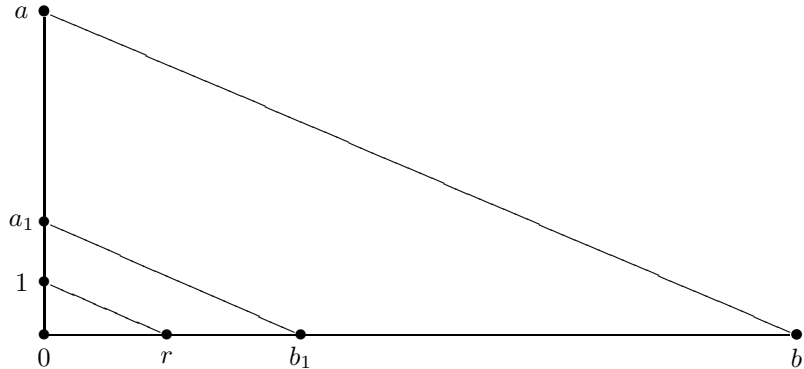
\[
	\xy
	(0,43)*{\bullet};(100,0)*{\bullet} **\dir{-};
	(0,15)*{\bullet};(34.1,0)*{\bullet} **\dir{-};
	(100,0)*{\bullet};(0,0)*{} **\dir{-}; 
	(0,0)*{};(0,43)*{\bullet} **\dir{-}; 
	(0,7)*{\bullet};(16.275,0)*{\bullet} **\dir{-}; 
	(-3,43)*{a};(16.275,-3)*{r};(100,-3)*{b};(-3,7)*{1};(-3,15)*{a_{1}};(34.1,-3)*{b_{1}};(0,-3)*{0};(0,0)*{\bullet};
	\endxy
\]
\\
 By definition of multiplication we have,\\ \hspace{.75in} $b_{1}=a_{1}r$ and $b=ar=k(a_{1}r)=kb_{1}$ \hspace{.51in}(1).\\Lastly by the Pythagorean Thm (we can use the Pythagorean Thm here since it can be proved using only area considerations),\\ \hspace{.75in}$a_{1}^{2}+b_{1}^{2}=c_{1}^{2}$ and $a^{2}+b^{2}=c^{2}$ \hspace{.5in} (2).\\ By (1) and (2) we must have $c^{2}=(kc_{1})^{2}$ or $c=kc_{1}$. Hence, claim follows.
\\
$(\Leftarrow)$\\
Suppose $a=ka_{1}, b=kb_{1}$ and $c=kc_{1}$. By the Pythagorean Thm, $a^{2}+b^{2}=c^{2}$ which implies $a_{1}^{2}+b_{1}^{2}=c_{1}^{2}$. Hence, $\Delta A_{1}B_{1}C_{1}$ must be a right triangle. By definition of multiplication, there exists a unique $r>0$ such that $ar=b$ and $\overline{(0,a),(b,0)}$ is parallel to $\overline{(0,1),(r,0)}$. Now $ar=b, a=ka_{1}$ and $b=kb_{1}$ implies $ra_{1}=b_{1}$. Therefore, by defintion of multiplication we have $\overline{(0,a_{1}),(b_{1},0)}$ parallel to $\overline{(0,1),(r,0)}$. Hence, the triangles must be similar.
\end{proof}
\begin{theorem}
If $\Delta ABC$ has sides $a\leq b\leq c$ then any triangle $\Delta A_{1}B_{1}C_{1}$  with sides $a_{1}\leq b_{1} \leq c_{1}$ is similar to $\Delta ABC$ if and only if there exists a $ k>0$ such that $a=ka_{1}, b=kb_{1}$ and $c=kc_{1}$.
\end{theorem}
\begin{proof}
Note any triangle can be cut into two right triangles. Hence claim follows by applying previous Lemma to each of the right triangles.
\end{proof}
\section{Conclusion}
Mathematicians in the time of Pythagoras lived in a world where magnitudes and numbers were not the same thing (see [3]). 
By introducing the geometric axiomatic of real number multiplication, we hope to rescue students and teachers from a similar disconnect, most recently manifested in arguments whether multiplication is repeated addition (see [4] and [5]). The lack of an easy visualization of multiplication (how to multiply two line segments), in contrast with addition, may account for some of the struggles that students encounter when they first make the transition from the whole numbers to the integers and rational numbers. 
Extending multiplication of whole numbers to, integers, rational and real numbers within the same model we hope will be of some pedagogical value. We believe that both prospective K-12 mathematics teachers and STEM students could especially benefit by being exposed to the material in this paper before graduating from university. 
\section{acknowledgments}
We would like to thank Dr. Roy Smith for pointing out how Theorem 4 could be proven using Propostion 37 in book one of Euclid's \textit{elements}. We would also like to thank Jonathan Crabtree for pointing out the similarities between Hibert's approach and ours, which we were unaware of. In addition, we would like to thank Dr. Keith Devlin, Dr. Art DiVito, Dr. Bob Stein, Dr. Oscar Chavez, Dr. David Klein, Dr. Hung-Hsi Wu, Dr. Paul Libbrecht, Dr. Davida Fischman,  Dr. Rebecca Reinger, Colin McAllister, and all the anonymous referees for reading over and commenting on various versions of this paper. This in turn has lead to overall improvements in the readability and presentation of the paper. 

\end{document}